\documentclass{article}
\usepackage{graphicx}
\usepackage{amsmath}
\usepackage{amsfonts}
\usepackage{amssymb}
\newtheorem{theorem}{Theorem}

\newtheorem{corollary}[theorem]{Corollary}

\newtheorem{definition}[theorem]{Definition}
\newtheorem{example}[theorem]{Example}

\newtheorem{proposition}[theorem]{Proposition}

\newtheorem{remark}[theorem]{Remark}

\newenvironment{proof}[1][Proof]{\textbf{#1.} }{\ \rule{0.5em}{0.5em}}

\begin{document}

\title{A moment problem for pseudo-positive definite functionals}
\author{Ognyan Kounchev and Hermann Render}
\maketitle
\begin{abstract}
A moment problem is presented for a class of signed measures which are termed
pseudo-positive. Our main result says that for every pseudo-positive definite
functional (subject to some reasonable restrictions) there exists a
representing pseudo-positive measure.

The second main result is a characterization of determinacy in the class of
equivalent pseudo-positive representation measures. Finally the corresponding
truncated moment problem is discussed.

\textbf{Key words: } Multidimensional moment problem, pseudopositive measures,
spherical harmonics, multidimensional numerical integration.

MSC 2000 classification: 43A32, 47A57, 65D32
\end{abstract}

\section{Introduction}

Let $\mathbb{C}\left[  x_{1},...,x_{d}\right]  $ denote the space of all
polynomials in $d$ variables with complex coefficients and let $T:\mathbb{C}%
\left[  x_{1},...,x_{d}\right]  \rightarrow\mathbb{C}$ be a linear functional.
The \emph{multivariate moment problem} asks for conditions on the functional
$T$ such that there exists a \emph{non-negative} measure $\mu$ on
$\mathbb{R}^{d}$ with
\begin{equation}
T\left(  P\right)  =\int_{\mathbb{R}^{d}}P\left(  x\right)  d\mu\left(
x\right)  \label{eqmoment}%
\end{equation}
for all $P\in\mathbb{C}\left[  x_{1},...,x_{d}\right]  .$ It is well known
that \emph{positive definiteness} of the functional $T$ is a necessary
condition which means that
\[
T\left(  P^{\ast}P\right)  \geq0\qquad\text{ for all }P\in\mathbb{C}\left[
x_{1},...,x_{d}\right]  ;
\]
here $P^{\ast}$ is the polynomial whose coefficients are the complex
conjugates of the coefficients of $P.$ By a theorem of Haviland, a necessary
and sufficient condition for the existence of a non-negative measure $\mu$
satisfying (\ref{eqmoment}) is the \emph{positivity }of the
functional\emph{\ }$T$, i.e. $P\left(  x\right)  \geq0$ for all $x\in
\mathbb{R}^{d}$ implies $T\left(  P\right)  \geq0$ for all $P\in
\mathbb{C}\left[  x_{1},...,x_{d}\right]  $, cf. \cite[p. 111]{Berg1987}. In
the case $d=1$ it is a classical fact that a functional $T$ is positive if and
only if it is positive-definite, which is proved by using the representation
of a non-negative polynomial as a sum of two squares of polynomials, cf.
\cite[Chapter 1, section 1.1]{Akhi65}. A counter-example of D. Hilbert shows
that a representation of a multivariate non-negative polynomial as a finite
sum of squares is in general not possible, cf. \cite{BCR84}. Many authors have
tried to find additional assumptions on the functional $T$ such that positive
definiteness and positivity become equivalent, see \cite{BCR84},
\cite{CuFi00}, \cite[p. 47]{Fugl83}, \cite{McGr80}, \cite{PuVa99},
\cite{Schm91a}, \cite{StSz98}.

In this paper we shall be concerned with a \emph{modified} moment problem
which arised in the investigation of a new cubature formula of Gau\ss-Jacobi
type for measures $\mu$ in the multivariate setting, see \cite{KR05},
\cite{KR05B},\cite{KR07}. In contrast to the classical multivariate moment
problem we allow the measures $\mu$ under consideration to be \emph{signed}
measures on $\mathbb{R}^{d}$. Our approach is based on the new notions of
pseudo-positive definite functionals $T$ and pseudo-positive signed measures
$\mu,$ to be explained below.

A cornerstone of our approach is the \emph{Gauss representation} of a
polynomial which we provide below.\ First we recall some definitions and
notations: Let $\left|  x\right|  =\sqrt{x_{1}^{2}+....+x_{d}^{2}}$ be the
euclidean norm and $\mathbb{S}^{d-1}:=\left\{  x\in\mathbb{R}^{d}:\left|
x\right|  =1\right\}  $ be the unit sphere. We shall write $x\in\mathbb{R}%
^{d}$ in spherical coordinates $x=r\theta$ with $\theta\in\mathbb{S}^{d-1}.$
Let $\mathcal{H}_{k}\left(  \mathbb{R}^{d}\right)  $ be the set of all
harmonic homogeneous complex-valued polynomials of degree $k.$ Then
$f\in\mathcal{H}_{k}\left(  \mathbb{R}^{d}\right)  $ is called a \emph{solid
harmonic} and the restriction of $f$ to $\mathbb{S}^{d-1}$ a \emph{spherical
harmonic} of degree $k$. Throughout the paper we shall assume that
$Y_{k,l}:\mathbb{R}^{d}\rightarrow\mathbb{R},$ $l=1,...,a_{k}:=\dim
\mathcal{H}_{k}\left(  \mathbb{R}^{d}\right)  ,$ is an orthonormal basis of
$\mathcal{H}_{k}\left(  \mathbb{R}^{d}\right)  $ with respect to the scalar
product $\left\langle f,g\right\rangle _{\mathbb{S}^{d-1}}:=\int
_{\mathbb{S}^{d-1}}f\left(  \theta\right)  \overline{g\left(  \theta\right)
}d\theta.$ We shall often use the trivial identity $Y_{k,l}\left(  x\right)
=r^{k}Y_{kl}\left(  \theta\right)  .$  The Gauss representation (cf.
\cite{ABR92}, \cite{StWe71} or \cite[Theorem 10.2]{Koun00}) tells us that for
every $P\in\mathbb{C}\left[  x_{1},...,x_{d}\right]  $ there exist polynomials
$p_{k,l}$ such that
\begin{equation}
P\left(  x\right)  =\sum_{k=0}^{\deg P}\sum_{l=1}^{a_{k}}p_{k,l}\left(
r^{2}\right)  \cdot r^{k}Y_{k,l}\left(  \theta\right)  =\sum_{k=0}^{\deg
P}\sum_{l=1}^{a_{k}}p_{k,l}\left(  \left|  x\right|  ^{2}\right)
Y_{k,l}\left(  x\right)  \label{gaussrepresentation}%
\end{equation}
where $\deg P$ is the degree of the polynomial $P.$ By this formula it is
clear that the set of polynomials
\[
\left\{  \left|  x\right|  ^{2j}Y_{k,l}\left(  x\right)  :j\geq0,k\geq
0,l=1,2,...,a_{k}\right\}
\]
forms a basis for the space of all polynomials, hence this is an alternative
basis to the standard basis $\left\{  x^{\alpha}:\alpha\in\mathbb{Z}%
^{d},\alpha\geq0\right\}  $. The numbers
\begin{equation}
c_{j,k,l}:=\int\left|  x\right|  ^{2j}Y_{k,l}\left(  x\right)  d\mu\left(
x\right)  \label{cjkl}%
\end{equation}
are sometimes called the \emph{distributed moments} of $\mu,$ cf.
\cite{butkovskii}, \cite{butkovskii2}, \cite{kounchev84}, \cite{kounchev85},
\cite{kounchev87}. Let us remark that for fixed $k,l$ one may consider the
correspondence $j\longmapsto c_{j,k,l}$ as a univariate moment sequence in the
variable $j\in\mathbb{N}_{0}.$ The distributed moments can be expressed
linearly by the classical \emph{monomial moments}
\begin{equation}
\int x^{\alpha}d\mu\left(  x\right)  \label{calfa}%
\end{equation}
which are considered in the standard approach, and vice versa.

Now we will introduce our basic notions: A signed measure $\mu$ over
$\mathbb{R}^{d}$ is \emph{pseudo-positive with respect to the orthonormal
basis} $Y_{k,l},l=1,...,a_{k}$, $k\in\mathbb{N}_{0}$ if the inequality
\begin{equation}
\int_{\mathbb{R}^{d}}h\left(  \left|  x\right|  \right)  Y_{k,l}\left(
x\right)  d\mu\left(  x\right)  \geq0\label{defpspos}%
\end{equation}
holds for every non-negative continuous function $h:\left[  0,\infty\right)
\rightarrow\left[  0,\infty\right)  $ with compact support, and for all
$k\in\mathbb{N}_{0}$ and $l=1,2,...,a_{k}.$ Obviously, the
\emph{radially-symmetric measures} represent a subclass of the pseudo-positive
measures 

Given a linear functional $T:$ $\mathbb{C}\left[  x_{1},...,x_{d}\right]
\rightarrow\mathbb{C}$ and $Y_{k,l}\in\mathcal{H}_{k}\left(  \mathbb{R}%
^{d}\right)  $ we define the ''component functional'' $T_{k,l}:\mathbb{C}%
\left[  x_{1}\right]  \rightarrow\mathbb{C}$ by putting
\begin{equation}
T_{k,l}\left(  p\right)  :=T\left(  p(\left|  x\right|  ^{2})Y_{k,l}\left(
x\right)  \right)  \qquad\text{ for every }p\in\mathbb{C}\left[  x_{1}\right]
.\label{defpseudo2}%
\end{equation}
Note that in the notations (\ref{cjkl}), $T_{k,l}\left(  p\right)  =c_{j,k,l}$
for $p\left(  t\right)  =t^{j}$ with  $j\in\mathbb{N}_{0}.$ We say that the
functional $T$ is \emph{pseudo-positive definite} \emph{with respect to the
orthonormal basis} $Y_{k,l},l=1,...,a_{k}$, $k\in\mathbb{N}_{0}$ if
\[
T_{k,l}\left(  p^{\ast}\left(  t\right)  p\left(  t\right)  \right)
\geq0\text{ and }T_{k,l}\left(  t\cdot p^{\ast}\left(  t\right)  p\left(
t\right)  \right)  \geq0
\]
for every $p\left(  t\right)  \in\mathbb{C}\left[  x_{1}\right]  ,$ and for
every $k\in\mathbb{N}_{0}$ and $l=1,...,a_{k}$. 

Our main result in Section \ref{Smomentproblem} provides  a reasonable
\emph{sufficient criterion} guaranteeing that for a pseudo-positive definite
functional $T:$ $\mathbb{C}\left[  x_{1},...,x_{d}\right]  \rightarrow
\mathbb{C}$ there exists a pseudo-positive signed measure $\mu$ on
$\mathbb{R}^{d}$ with
\begin{equation}
\int_{\mathbb{R}^{d}}P\left(  x\right)  d\mu=T\left(  P\right)  \text{ for all
}P\in\mathbb{C}\left[  x_{1},...,x_{d}\right]  .\label{eqpoll}%
\end{equation}
This means that we give a solution to the \emph{pseudo-positive moment
problem}: this problem asks  for conditions on the moments (\ref{cjkl}) which
provide the existence of a pseudo-positive (signed) measure $\mu$ satisfying
the equalities  (\ref{cjkl}).The sufficient criterion is a \emph{summability
assumption} of the type
\begin{equation}
\sum_{k=0}^{\infty}\sum_{l=1}^{a_{k}}\int_{0}^{\infty}r^{N}r^{-k}d\sigma
_{k,l}\left(  r\right)  <\infty\text{ for all }N\in\mathbb{N}_{0}%
\label{maincond2}%
\end{equation}
where the measures $\sigma_{k,l}$ are representing measures of the component
functionals $T_{k,l},$ cf. Proposition \ref{PropStieltjes}. 

An essential advantage of our approach is that there exists a naturally
defined \emph{truncated moment problem} in the class of pseudo-positive
definite functionals. In Section \ref{STruncated} we shall formulate and solve
this problem which is important also from practical point of view. 

The second main result in Section \ref{Sdeterminacy} says that the
pseudo-positive representing measure $\mu$ of a pseudo-positive definite
functional $T:\mathbb{C}\left[  x_{1},...,x_{d}\right]  \rightarrow\mathbb{C}$
is unique in the class of all pseudo-positive signed measures whenever each
functional $T_{k,l}$ defined in (\ref{defpspos}) has a unique representing
measure on $\left[  0,\infty\right)  $ in the sense of Stieltjes (for the
precise definition see Section \ref{Sdeterminacy}). And vice versa, if a
pseudo--positive functional $T$ is determinate in the class of all
pseudo-positive signed measures and the summability condition (\ref{maincond2}%
) is satisfied, then each functional $T_{k,l}$ is determinate in the sense of
Stieltjes. The proof is essentially based on the properties of the Nevanlinna
extremal measures. In the last Section we shall give examples and some further
properties of pseudo-positive definite functionals.

Let us recall some terminology from measure theory: a signed measure on
$\mathbb{R}^{d}$ is a set function on the Borel $\sigma$-algebra on
$\mathbb{R}^{d}$ which takes real values and is $\sigma$-additive. For the
standard terminology, as Radon measure, Borel $\sigma$-algebra, etc., we refer
to \cite{BCR84}. By the \emph{Jordan decomposition } \cite[p. 125]{Cohn}, a
signed measure $\mu$ is the difference of two non-negative finite measures,
say $\mu=\mu^{+}-\mu^{-}$ with the property that there exist a Borel set $A$
such that $\mu^{+}\left(  A\right)  =0$ and $\mu^{-}\left(  \mathbb{R}%
^{n}\setminus A\right)  =0.$ The \emph{variation} of $\mu$ is defined as
$\left|  \mu\right|  :=\mu^{+}+\mu^{-}.$ The signed measure $\mu$ is called
\emph{moment measure} if all polynomials are integrable with respect to
$\mu^{+}$ and $\mu^{-},$ which is equivalent to integrability with respect to
the total variation. The \emph{support of a non-negative measure} $\mu$ on
$\mathbb{R}^{d}$ is defined as the complement of the largest open set $U$ such
that $\mu\left(  U\right)  =0.$ In particular, the \emph{support of the zero
measure} is the \emph{empty set}. The \emph{support of a signed measure}
$\sigma$ is defined as the support of the total variation $\left|
\sigma\right|  =\sigma_{+}+\sigma_{-}$ (see \cite[p. 226]{Cohn}). Recall that
in general, the supports of $\sigma_{+}$ and $\sigma_{-}$ are not disjoint
(cf. exercise 2 in \cite[p. 231]{Cohn}). For a surjective measurable mapping
$\varphi:X\rightarrow Y$ and a measure $\nu$ on $X$ the \emph{image measure}
$\nu^{\varphi}$ on $Y$ is defined by
\begin{equation}
\nu^{\varphi}\left(  B\right)  :=\nu\left(  \varphi^{-1}B\right)
\label{imagemeasure1}%
\end{equation}
for all Borel subsets $B$ of $Y.$ The equality $\int_{X}g\left(
\varphi\left(  x\right)  \right)  d\nu\left(  x\right)  =\int_{Y}g\left(
y\right)  d\nu^{\varphi}\left(  y\right)  $ holds for all integrable functions
$g$.

\section{The moment problem for pseudo-positive definite
functionals\label{Smomentproblem}}

Recall that for a continuous function $f:\mathbb{R}^{d}\rightarrow\mathbb{C}$
the \emph{Laplace--Fourier coefficient }is defined by
\begin{equation}
f_{k,l}\left(  r\right)  =\int_{\mathbb{S}^{d-1}}f\left(  r\theta\right)
Y_{k,l}\left(  \theta\right)  d\theta.\label{eqfourier}%
\end{equation}
The formal expansion
\begin{equation}
f\left(  r\theta\right)  =\sum_{k=0}^{\infty}\sum_{l=1}^{a_{k}}f_{k,l}\left(
r\right)  Y_{k,l}\left(  \theta\right)  \label{frtheta}%
\end{equation}
is the \emph{Laplace--Fourier series. }The following result may be found e.g.
in \cite{Baouendi} or \cite{Sob}.

\begin{proposition}
The Laplace-Fourier coefficient $f_{k,l}$ of a polynomial $f$ given by
(\ref{eqfourier}) is of the form $f_{k,l}\left(  r\right)  =r^{k}%
p_{k,l}\left(  r^{2}\right)  $ where $p_{k,l}$ is a univariate polynomial.
Hence, the Laplace-Fourier series (\ref{frtheta}) is equal to
\begin{equation}
f\left(  x\right)  =\sum_{k=0}^{\deg f}\sum_{l=1}^{a_{k}}p_{k,l}(\left|
x\right|  ^{2})Y_{k,l}\left(  x\right)  . \label{gauss}%
\end{equation}
\end{proposition}

The next two Propositions characterize pseudo-positive definite functionals:

\begin{proposition}
\label{PropStieltjes} Let $T:\mathbb{C}\left[  x_{1},...,x_{d}\right]
\rightarrow\mathbb{C}$ be a pseudo-positive definite functional. Then for each
$k\in\mathbb{N}_{0},$ and  $l=1,...,a_{k},$ there exist non-negative measures
$\sigma_{k,l}$ with support in $\left[  0,\infty\right)  $ such that
\begin{equation}
T\left(  f\right)  =\sum_{k=0}^{\deg f}\sum_{l=1}^{a_{k}}\int_{0}^{\infty
}f_{k,l}\left(  r\right)  r^{-k}d\sigma_{k,l}\left(  r\right)
\label{eqTnicerep}%
\end{equation}
holds for all $f\in\mathbb{C}\left[  x_{1},...,x_{d}\right]  $ where
$f_{k,l}\left(  r\right)  $, $k\in\mathbb{N}_{0},$ $l=1,...,a_{k},$ are the
Laplace-Fourier coefficients of $f.$
\end{proposition}

\begin{proof}
By the solution of the Stieltjes moment problem there exists a non-negative
measure $\mu_{k,l}$ with support in $\left[  0,\infty\right)  $ representing
the functional $T_{k,l},$ i.e. satisfying
\begin{equation}
T_{k,l}\left(  p\right)  =\int_{0}^{\infty}p\left(  t\right)  d\mu
_{k,l}\left(  t\right)  \qquad\text{for every }p\in\mathbb{C}\left[  t\right]
. \label{eqT1}%
\end{equation}
Let now $\varphi:\left[  0,\infty\right)  \rightarrow\left[  0,\infty\right)
$ be defined by $\varphi\left(  t\right)  =\sqrt{t}.$ Then we put
$\sigma_{k,l}:=\mu_{k,l}^{\varphi}$ where $\mu_{k,l}^{\varphi}$ is the image
measure defined in (\ref{imagemeasure1}). We obtain
\begin{equation}
\int_{0}^{\infty}h\left(  t\right)  d\mu_{k,l}\left(  t\right)  =\int
_{0}^{\infty}h\left(  r^{2}\right)  d\mu_{k,l}^{\varphi}\left(  r\right)  .
\label{eqT2}%
\end{equation}
Now use (\ref{gauss}), the linearity of $T$ and the definition of $T_{k,l}$ in
(\ref{defpseudo2}), and the equations (\ref{eqT1}) and (\ref{eqT2}) to obtain
\[
T\left(  f\right)  =\sum_{k=0}^{\deg f}\sum_{l=1}^{a_{k}}T_{k,l}\left(
p_{k,l}\right)  =\sum_{k=0}^{\deg f}\sum_{l=1}^{a_{k}}\int_{0}^{\infty}%
p_{k,l}\left(  r^{2}\right)  d\mu_{k,l}^{\varphi}\left(  r\right)  .
\]
Since $p_{k,l}\left(  r^{2}\right)  =r^{-k}f_{k,l}\left(  r\right)  $ the
claim (\ref{eqTnicerep}) follows from the last equation, which ends the proof.
\end{proof}

The next result shows that the converse of Proposition \ref{PropStieltjes} is
also true; not less important, it is a natural way of defining pseudo-positive
definite functionals.

\begin{proposition}
\label{PropTTT}Let $\sigma_{k,l},$ $k\in\mathbb{N}_{0},$ $l=1,...,a_{k},$ be
non-negative moment measures with support in $\left[  0,\infty\right)  .$ Then
the functional $T:\mathbb{C}\left[  x_{1},...,x_{d}\right]  \rightarrow
\mathbb{C}$ defined by
\begin{equation}
T\left(  f\right)  :=\sum_{k=0}^{\deg f}\sum_{l=1}^{a_{k}}\int_{0}^{\infty
}f_{k,l}\left(  r\right)  r^{-k}d\sigma_{k,l} \label{defTTT}%
\end{equation}
is pseudo-positive definite, where $f_{k,l}\left(  r\right)  $, $k\in
\mathbb{N}_{0},$ $l=1,...,a_{k},$ are the Laplace-Fourier coefficients of $f.$
\end{proposition}

\begin{proof}
Let us compute $T_{k,l}\left(  p\right)  $ where $p$ is a univariate
polynomial: by definition, $T_{k,l}\left(  p\right)  =T\left(  p(\left|
x\right|  ^{2})Y_{k,l}\left(  x\right)  \right)  $. The Laplace-Fourier series
of the function $x\mapsto\left|  x\right|  ^{2j}p(\left|  x\right|
^{2})Y_{k,l}\left(  x\right)  $ is equal to $r^{2j}p\left(  r^{2}\right)
r^{k}Y_{k,l}\left(  \theta\right)  $, hence
\[
T_{k,l}\left(  t^{j}p\left(  t\right)  \right)  =T\left(  \left|  x\right|
^{2j}p(\left|  x\right|  ^{2})Y_{k,l}\left(  x\right)  \right)  =\int
_{0}^{\infty}r^{j}p\left(  r^{2}\right)  d\sigma_{k,l}%
\]
for every natural number $j.$ Taking $j=0$ and $j=1$ one concludes that
$T_{k,l}\left(  p^{\ast}\left(  t\right)  p\left(  t\right)  \right)  \geq0$
and $T_{k,l}\left(  tp^{\ast}\left(  t\right)  p\left(  t\right)  \right)
\geq0$ for all univariate polynomials $p$, hence $T$ is pseudo-positive definite.
\end{proof}

By $C\left(  X\right)  $ we denote the space of all continuous complex-valued
functions on a topological space $X$ while $C_{c}\left(  X\right)  $ is the
set of all $f\in C\left(  X\right)  $ having compact support. Further
$C_{pol}\left(  \mathbb{R}^{d}\right)  $ is the space of all polynomially
bounded, continuous functions, so for each $f\in C_{pol}\left(  \mathbb{R}%
^{d}\right)  $ there exists $N\in\mathbb{N}_{0},$ such that $\left|  f\left(
x\right)  \right|  \leq C_{N}\left(  1+\left|  x\right|  \right)  ^{N}$ for
some constant $C_{N}$ (depending on $f$ ) for all $x\in\mathbb{R}^{d}.$ A
useful space of test functions is
\begin{equation}
C^{\times}\left(  \mathbb{R}^{d}\right)  :=\{\sum_{k=0}^{N}\sum_{l=1}^{a_{k}%
}f_{k,l}\left(  \left|  x\right|  \right)  Y_{k,l}\left(  x\right)
:N\in\mathbb{N}_{0}\text{ and }f_{k,l}\in C\left[  0,\infty\right)  \}.
\label{Ccross}%
\end{equation}
which can be rephrased as the set of all continuous functions with a finite
Laplace-Fourier series.

\begin{proposition}
\label{pseudopos}Let $\mu$ be a pseudo-positive moment measure on
$\mathbb{R}^{d}.$ Then there exist unique moment measures $\mu_{k,l}$ defined
on $\left[  0,\infty\right)  $ such that
\begin{equation}
\int_{0}^{\infty}h\left(  t\right)  d\mu_{k,l}\left(  t\right)  =\int
_{\mathbb{R}^{d}}h\left(  \left|  x\right|  \right)  Y_{k,l}\left(  x\right)
d\mu\label{eqlim}%
\end{equation}
holds for all $h\in C_{pol}\left[  0,\infty\right)  $. Further for each $f\in
C^{\times}\left(  \mathbb{R}^{d}\right)  \cap C_{pol}\left(  \mathbb{R}%
^{d}\right)  $
\[
\int_{\mathbb{R}^{d}}f\left(  x\right)  d\mu=\sum_{k=0}^{\infty}\sum
_{l=1}^{a_{k}}\int_{0}^{\infty}f_{k,l}\left(  r\right)  r^{-k}d\mu_{k,l}.
\]
\end{proposition}

\begin{proof}
By definition of pseudo-positivity, $M_{k,l}\left(  h\right)  :=\int
_{\mathbb{R}^{d}}h\left(  \left|  x\right|  \right)  Y_{k,l}\left(  x\right)
d\mu$ defines a positive functional on $C_{c}\left(  \left[  0,\infty\right)
\right)  .$ By the Riesz representation theorem there exists a unique
non-negative measure $\mu_{k,l}$ such that $M_{k,l}\left(  h\right)  =\int
_{0}^{\infty}h\left(  t\right)  d\mu_{k,l}$ for all $h\in C_{c}\left(  \left[
0,\infty\right)  \right)  .$ We want to show that (\ref{eqlim}) holds for all
$h\in C_{pol}\left[  0,\infty\right)  $. For this, let $u_{R}:\left[
0,\infty\right)  \rightarrow\left[  0,1\right]  $ be a \emph{cut--off
function, }so $u_{R}$ is continuous and decreasing such that
\begin{equation}
u_{R}\left(  r\right)  =1\text{ for all }0\leq r\leq R\text{ and }u_{R}\left(
r\right)  =0\text{ for all }r\geq R+1. \label{defuR}%
\end{equation}
Let $h\in C_{pol}\left[  0,\infty\right)  .$ Then $u_{R}h\in C_{c}\left(
\left[  0,\infty\right)  \right)  $ and
\begin{equation}
\int_{0}^{\infty}u_{R}\left(  t\right)  h\left(  t\right)  d\mu_{k,l}%
=\int_{\mathbb{R}^{d}}u_{R}\left(  \left|  x\right|  \right)  h\left(  \left|
x\right|  \right)  Y_{k,l}\left(  x\right)  d\mu. \label{eqlimitR}%
\end{equation}
Note that $\left|  u_{R}\left(  t\right)  h\left(  t\right)  \right|
\leq\left|  u_{R+1}\left(  t\right)  h\left(  t\right)  \right|  $ for all
$t\in\left[  0,\infty\right)  .$ Hence by the monotone convergence theorem
\begin{equation}
\int_{0}^{\infty}\left|  h\left(  t\right)  \right|  d\mu_{k,l}=\lim
_{R\rightarrow\infty}\int_{0}^{\infty}\left|  u_{R}\left(  t\right)  h\left(
t\right)  \right|  d\mu_{k,l}. \label{eqlimitR2}%
\end{equation}
On the other hand, it is obvious that
\begin{equation}
\left|  \int_{\mathbb{R}^{d}}u_{R}\left(  \left|  x\right|  \right)  \left|
h\left(  \left|  x\right|  \right)  \right|  Y_{k,l}\left(  x\right)
d\mu\right|  \leq\int_{\mathbb{R}^{d}}\left|  h\left(  \left|  x\right|
\right)  Y_{k,l}\left(  x\right)  \right|  d\left|  \mu\right|  . \label{eq36}%
\end{equation}
The last expression is finite since $\mu$ is a moment measure. From
(\ref{eqlimitR2}), (\ref{eqlimitR}) applied to $\left|  h\right|  $ and
(\ref{eq36}) it follows that $\left|  h\right|  $ is integrable for $\mu
_{k,l}.$ Using Lebesgue's convergence theorem for $\mu$ and (\ref{eqlimitR})
it is easy to that (\ref{eqlim}) holds. For the last statement recall that
each $f\in C^{\times}\left(  \mathbb{R}^{d}\right)  $ has a finite
Laplace-Fourier series, and it is easy to see that the Laplace-Fourier
coefficients $f_{k,l}$ are in $C_{pol}\left[  0,\infty\right)  $ if $f\in
C_{pol}\left(  \mathbb{R}^{d}\right)  $, see (\ref{eqffpol}) below.
\end{proof}

The next theorem is the main technical result of this section. 

\begin{theorem}
\label{ThmRepG}Let $\sigma_{k,l},$ $k\in\mathbb{N}_{0},$ $l=1,...,a_{k},$ be
non-negative measures with support in $\left[  0,\infty\right)  $ such that
for any $N\in\mathbb{N}_{0}$
\begin{equation}
C_{N}:=\sum_{k=0}^{\infty}\sum_{l=1}^{a_{k}}\int_{0}^{\infty}r^{N}%
r^{-k}d\sigma_{k,l}<\infty\text{ .} \label{neu21}%
\end{equation}
Then for the functional $T:\mathbb{C}\left[  x_{1},...,x_{d}\right]
\rightarrow\mathbb{C}$ defined by (\ref{defTTT}) there exists a
pseudo-positive, signed moment measure $\sigma$ such that
\[
T\left(  f\right)  =\int_{\mathbb{R}^{n}}fd\sigma\text{ for all }%
f\in\mathbb{C}\left[  x_{1},...,x_{d}\right]  .
\]
\end{theorem}

\begin{remark}
\label{Rrep} 1. If the measures $\sigma_{k,l}$ have supports in the compact
interval $\left[  \rho,R\right]  $ for all $k\in\mathbb{N}_{0},$
$l=1,...,a_{k},$ then the measure $\sigma$ in Theorem \ref{ThmRepG} has
support in the annulus $\left\{  x\in\mathbb{R}^{d}:\rho\leq\left|  x\right|
\leq R\right\}  .$

2. In the case of $R<\infty$ , it obviously suffices to assume that
$C_{0}<\infty$ instead of $C_{N}<\infty$ for all $N\in\mathbb{N}_{0}.$

3. The proof of Theorem \ref{ThmRepG} shows that $\sigma_{k,l}$ is equal to
the measure induced by $\sigma$ with respect to the solid harmonic
$Y_{k,l}\left(  x\right)  ,$ cf. (\ref{defpspos}).
\end{remark}

\begin{proof}
1. We show at first that $T$ can be extended to a linear functional
$\widetilde{T}$ defined on $C_{pol}\left(  \mathbb{R}^{d}\right)  $ by the
formula
\begin{equation}
\widetilde{T}\left(  f\right)  :=\sum_{k=0}^{\infty}\sum_{l=1}^{a_{k}}\int
_{0}^{\infty}f_{k,l}\left(  r\right)  r^{-k}d\sigma_{k,l} \label{eqtschlange}%
\end{equation}
for $f\in C_{pol}\left(  \mathbb{R}^{d}\right)  ,$ where $f_{k,l}\left(
r\right)  $ are the Laplace-Fourier coefficients of $f$. Indeed, since $f\in
C_{pol}\left(  \mathbb{R}^{d}\right)  $ is of polynomial growth there exists
$C>0$ and $N\in\mathbb{N}$ such that $\left|  f\left(  x\right)  \right|  \leq
C(1+\left|  x\right|  ^{N})$. Let $\omega_{d-1}$ denote the surface area of
the unit sphere. It follows from (\ref{eqfourier}) that
\begin{equation}
\left|  f_{k,l}\left(  r\right)  \right|  \leq C\left(  1+r^{N}\right)
\sqrt{\omega_{d-1}}\sqrt{\int_{\mathbb{S}^{d-1}}\left|  Y_{k,l}\left(
\theta\right)  \right|  ^{2}d\theta}=C\left(  1+r^{N}\right)  \sqrt
{\omega_{d-1}}, \label{eqffpol}%
\end{equation}
where we used the Cauchy-Schwarz inequality and the fact that $Y_{k,l}$ is
orthonormal. Hence,
\[
\int_{0}^{\infty}\left|  f_{k,l}\left(  r\right)  \right|  r^{-k}d\sigma
_{k,l}\leq\sqrt{\omega_{d-1}}C\int_{0}^{\infty}\left(  1+r^{N}\right)
r^{-k}d\sigma_{k,l}.
\]
By assumption (\ref{neu21}) the latter integral exists, so $f_{k,l}\left(
r\right)  r^{-k}$ is integrable with respect to $\sigma_{k,l}.$ By summing
over all $k,l$ we obtain by (\ref{neu21}) that
\[
\sum_{k=0}^{\infty}\sum_{l=1}^{a_{k}}\left|  \int_{0}^{\infty}f_{k,l}\left(
r\right)  r^{-k}d\sigma_{k,l}\right|  <\infty,
\]
which implies the convergence of the series in (\ref{eqtschlange}). It follows
that $\widetilde{T}$ is well-defined.

2. Let $T_{0}$ be the restriction of the functional $\widetilde{T}$ to the
space $C_{c}\left(  \mathbb{R}^{d}\right)  $. We will show that $T_{0}$ is
continuous. Let $f\in C_{c}\left(  \mathbb{R}^{d}\right)  $ and suppose that
$f $ has support in the annulus $\left\{  x\in\mathbb{R}^{d}:\rho\leq\left|
x\right|  \leq R\right\}  $ (for the case $\rho=0$ this is a ball). Then by a
similar technique as above $\left|  f_{k,l}\left(  r\right)  \right|
\leq\sqrt{\omega_{d-1}}\max_{\rho\leq\left|  x\right|  \leq R}\left|  f\left(
x\right)  \right|  .$ Using (\ref{eqtschlange}) one arrives at
\begin{equation}
\left|  T_{0}\left(  f\right)  \right|  \leq\max_{\rho\leq\left|  x\right|
\leq R}\left|  f\left(  x\right)  \right|  \sqrt{\omega_{d-1}}\sum
_{k=0}^{\infty}\sum_{l=1}^{a_{k}}\int_{\rho}^{R}r^{-k}d\sigma_{k,l}.
\label{eqkey23}%
\end{equation}

3. First consider the case that all measures $\sigma_{k,l}$ have supports in
the interval $\left[  \rho,R\right]  $ with $R<\infty$ (cf. Remark
\ref{Rrep}). Then (\ref{eqkey23}) and the Riesz representation theorem for
compact spaces yield a representing measure $\mu$ with support in the annulus
$\left\{  x\in\mathbb{R}^{d}:\rho\leq\left|  x\right|  \leq R\right\}  .$
Clearly $\mu$ is a moment measure. The pseudo--positivity of $\mu$ will be
proved in item 5.) below.

4. In the case that $\sigma_{k,l}$ have supports in $\left[  0,\infty\right)
$, we apply the Riesz representation theorem given in \cite[p. 41, Theorem
2.5]{BCR84}: there exists a unique signed measure $\sigma$ such that
$T_{0}\left(  g\right)  =\int_{\mathbb{R}^{d}}gd\sigma$ for all $g\in
C_{c}\left(  \mathbb{R}^{d}\right)  .$ Next we will show that the polynomials
are integrable with respect to the variation of the representation measure
$\sigma.$ Let $\sigma=\sigma_{+}-\sigma_{-}$ be the Jordan decomposition of
$\sigma$. Following the techniques of Theorem 2.4 and Theorem 2.5 in \cite[p.
42]{BCR84}, we have the equality
\begin{equation}
\int_{\mathbb{R}^{d}}g\left(  x\right)  d\sigma_{+}=\sup\left\{  T_{0}\left(
h\right)  :h\in C_{c}\left(  \mathbb{R}^{d}\right)  \text{ with }0\leq h\leq
g\right\}  \label{defTnull}%
\end{equation}
which holds for any non-negative function $g\in C_{c}\left(  \mathbb{R}%
^{d}\right)  .$ Let $u_{R}$ be the cut-off function defined in (\ref{defuR}).
We want to estimate $\int_{\mathbb{R}^{d}}g\left(  x\right)  d\sigma_{+}$ for
the function $g:=\left|  x\right|  ^{N}u_{R}(\left|  x\right|  ^{2}).$ In view
of (\ref{defTnull}), let $h\in C_{c}\left(  \mathbb{R}^{d}\right)  $ with
$0\leq h\left(  x\right)  \leq\left|  x\right|  ^{N}u_{R}(\left|  x\right|
^{2})$ for all $x\in\mathbb{R}^{d}.$ Then for the Laplace-Fourier coefficient
$h_{k,l}$ of $h$ we have the estimate
\[
\left|  h_{k,l}\left(  r\right)  \right|  \leq\sqrt{\int_{\mathbb{S}^{d-1}%
}\left|  h\left(  r\theta\right)  \right|  ^{2}d\theta}\sqrt{\int
_{\mathbb{S}^{d-1}}\left|  Y_{k,l}\left(  \theta\right)  \right|  ^{2}d\theta
}\leq r^{N}u_{R}\left(  r^{2}\right)  \sqrt{\omega_{d-1}}.
\]
According to (\ref{eqtschlange})
\[
T_{0}\left(  h\right)  \leq\left|  T_{0}\left(  h\right)  \right|  \leq
\sqrt{\omega_{d-1}}\sum_{k=0}^{\infty}\sum_{l=1}^{a_{k}}\int_{0}^{\infty}%
r^{N}r^{-k}d\sigma_{k,l}=:D_{N}.
\]
From (\ref{defTnull}) it follows that $\int_{\mathbb{R}^{d}}\left|  x\right|
^{N}u_{R}(\left|  x\right|  ^{2})d\sigma_{+}\leq D_{N}$ for all $R>0$ (note
that $D_{N}$ does not depend on $R$ ). By the monotone convergence theorem
(note that $u_{R}\left(  x\right)  \leq u_{R+1}\left(  x\right)  $ for all
$x\in\mathbb{R}^{d})$ we obtain
\[
\int_{\mathbb{R}^{d}}\left|  x\right|  ^{N}d\sigma_{+}=\lim_{R\rightarrow
\infty}\int_{\mathbb{R}^{d}}\left|  x\right|  ^{N}u_{R}(\left|  x\right|
^{2})d\sigma_{+}\leq D_{N}.
\]
Similarly one shows that $\int_{\mathbb{R}^{d}}\left|  x\right|  ^{N}%
d\sigma_{-}<\infty$ by considering the functional $S=-T_{0}$. It follows that
all polynomials are integrable with respect to $\sigma_{+}$ and $\sigma_{-}$.
Using similar arguments it is not difficult to see that for all $g\in
C^{\times}\left(  \mathbb{R}^{d}\right)  \cap C_{pol}\left(  \mathbb{R}%
^{d}\right)  $
\begin{equation}
\int_{\mathbb{R}^{d}}g\left(  x\right)  d\sigma=\widetilde{T}\left(  g\right)
.\text{ } \label{eqidentmt}%
\end{equation}

5. It remains to prove that $\sigma$ is pseudo-positive. Let $h\in
C_{c}\left(  \left[  0,\infty\right)  \right)  $ be a non-negative function.
The Laplace-Fourier coefficients $f_{k^{\prime},l^{\prime}}$ of $f\left(
x\right)  :=h\left(  \left|  x\right|  \right)  Y_{k,l}\left(  x\right)  $ are
given by $f_{k^{\prime}l^{\prime}}\left(  r\right)  =\delta_{kk^{\prime}%
}\delta_{ll^{\prime}}h\left(  r\right)  r^{k}$ and by (\ref{eqidentmt}) it
follows that
\[
\int_{\mathbb{R}^{d}}h\left(  \left|  x\right|  \right)  Y_{k,l}\left(
x\right)  d\sigma=\widetilde{T}\left(  f\right)  =\int_{0}^{\infty}%
f_{k,l}\left(  r\right)  r^{-k}d\sigma_{k,l}=\int_{0}^{\infty}h\left(
r\right)  d\sigma_{k,l}.
\]
Since $\sigma_{k,l}$ are non-negative measures, the last term is non-negative,
thus $\sigma$ is pseudo-positive. The proof is complete.
\end{proof}

The following Theorem  is the main result of the present Section and   is an
immediate consequence of Theorem \ref{ThmRepG}. It provides a simple
sufficient condition for the pseudo-positive definite functional on
$\mathbb{C}\left[  x_{1},...,x_{d}\right]  $ defined in (\ref{defTTT}) to
possess a pseudo--positive representing measure. Let us note that not every
pseudo-positive definite functional has a pseudo-positive representing
measure, see Theorem \ref{Tnonexistence}.

\begin{theorem}
\label{CorMoment}Let $T:\mathbb{C}\left[  x_{1},...,x_{d}\right]
\rightarrow\mathbb{C}$ be a pseudo-positive definite functional. Let
$\sigma_{k,l},k\in\mathbb{N}_{0},$ $l=1,...,a_{k},$ be non-negative measures
with supports in $\left[  0,\infty\right)  $ representing the functional $T$
as obtained in Proposition \ref{PropStieltjes}. If for any $N\in\mathbb{N}%
_{0}$
\begin{equation}
\sum_{k=0}^{\infty}\sum_{l=1}^{a_{k}}\int_{0}^{\infty}r^{N}r^{-k}d\sigma
_{k,l}<\infty, \label{eqneuCN}%
\end{equation}
then there exists a pseudo-positive, signed moment measure $\sigma$ such that
\[
T\left(  f\right)  =\int fd\sigma\qquad\text{ for all }f\in\mathbb{C}\left[
x_{1},...,x_{d}\right]  .
\]
\end{theorem}

It would be interesting to see whether the summability condition
(\ref{eqneuCN}) may be weakened, cf. also the discussion at the end of Section
\ref{Sexamples}.

By the uniqueness of the representing measure in the Riesz representation
theorem for compact spaces we conclude from Theorem \ref{ThmRepG}:

\begin{corollary}
\label{ppp}Let $\mu$ be a signed measure with compact support. Then $\mu$ is
pseudo-positive if and only if $\mu$ is pseudo-positive definite as a
functional on $\mathbb{C}\left[  x_{1},...,x_{d}\right]  .$
\end{corollary}

Let us remark that Corollary \ref{ppp} does not hold without the compactness
assumption which follows from well known arguments in the univariate case:
Indeed, let $\nu_{1}$ be a non-negative moment measure on $\left[
0,\infty\right)  $ which is not determined in the sense of Stieltjes; hence
there exists a non-negative moment measure $\nu_{2}$ on $\left[
0,\infty\right)  $ such that $\nu_{1}\left(  p\right)  =\nu_{2}\left(
p\right)  $ for all univariate polynomials. Since $\nu_{1}\neq\nu_{2}$ there
exists a continuous function $h:\left[  0,\infty\right)  \rightarrow\left[
0,\infty\right)  $ with compact support that $\nu_{1}\left(  h\right)  \neq
\nu_{2}\left(  h\right)  .$ Without loss of generality assume that
\begin{equation}
\int_{0}^{\infty}h\left(  r\right)  d\nu_{1}-\int_{0}^{\infty}h\left(
r\right)  d\nu_{2}<0. \label{eqverschieden}%
\end{equation}
For $i=1,2$ define $\mu_{i}=d\theta d\nu_{i},$ so for any $f\in C\left(
\mathbb{R}^{d}\right)  $ of polynomial growth
\[
\int fd\mu_{i}=\int_{0}^{\infty}\int_{\mathbb{S}^{d-1}}f\left(  r\theta
\right)  d\theta d\nu_{i}.
\]
For a polynomial $f$ let $f_{0}$ be the first Laplace--Fourier coefficient.
Then $\int fd\mu_{i}=\int_{0}^{\infty}f_{0}\left(  r\right)  d\nu_{i}$ for
$i=1,2.$ Since $\nu_{1}\left(  p\right)  =\nu_{2}\left(  p\right)  $ for all
univariate polynomials it follows that $\int fd\mu_{1}=\int fd\mu_{2}$ for all
polynomials. Then $\mu:=\mu_{1}-\mu_{2}$ is a signed measure which is
pseudo-positive definite since $\mu\left(  P\right)  =0$ for all polynomials
$P.$ It is not pseudo-positive since $\mu_{0}\left(  h\right)  =\int h\left(
\left|  x\right|  \right)  d\mu<0$ by (\ref{eqverschieden}).

\section{\label{STruncated}The truncated moment problem for pseudo-positive
definite functionals}

The classical \emph{truncated moment problem} of order $2n-1$ for a sequence
of real numbers $s_{0},s_{1},s_{2},...$ asks for conditions providing the
existence of a non-negative measure $\sigma_{n}$ on the real line such that
\begin{equation}
s_{k}=\int_{-\infty}^{\infty}t^{k}d\sigma_{n}\left(  t\right)  \text{ for
}k=0,...,2n-1,\label{skmom}%
\end{equation}
cf. \cite[p. 30]{Akhi65}. Let $\mathcal{P}_{\leq m}$ denote the space of all
univariate polynomials of degree $\leq m,$ and let us associate to the numbers
$s_{0},...,s_{2n}$ the linear functional $T_{n}:\mathcal{P}_{\leq
2n}\rightarrow\mathbb{R}$ defined by
\[
T_{n}\left(  t^{k}\right)  :=s_{k}\text{ for }k=0,...,2n.
\]
A necessary and sufficient condition for the existence of a non-negative
measure $\sigma_{n}$ on the real line satisfying (\ref{skmom}) is that $T_{n}$
is \emph{positive definite on }$\mathcal{P}_{\leq2n}$which means that
\[
T_{n}\left(  p^{\ast}\left(  t\right)  p\left(  t\right)  \right)  \geq0\text{
for all }p\in\mathcal{P}_{\leq n},
\]
see \cite[p. 30]{Akhi65}. Moreover, if $T_{n}$ is strictly positive definite
on $\mathcal{P}_{\leq2n}$ (i.e. that $T_{n}\left(  p^{\ast}\left(  t\right)
p\left(  t\right)  \right)  >0$ for all $p\in\mathcal{P}_{\leq n},p\neq0)$
then one can find a whole continuum of solutions to the truncated problem of
order $2n-1.$ 

A classical argument based on the Helly theorem shows that the  solutions
$\sigma_{n}$ of the truncated moment problem of order $2n-1$ for
$n\in\mathbb{N}_{0}$ converge to a solution $\sigma$ of the moment problem.
For a discussion of  truncated multivariate moment problems we refer to
\cite{CuFi00} and \cite{Stoc01}.

We now formulate a truncated moment problem in our framework. A basic question
is of course which moments are assumed to be known. Our formulation will
depend on two parameters, namely $n\in\mathbb{N}_{0}$ and $k_{0}\in
\mathbb{N}_{0}\cup\left\{  \infty\right\}  .$ We define the space
$U_{n}\left(  k_{0}\right)  $ as the set of all polynomials $f\in
\mathbb{C}\left[  x_{1},...,x_{d}\right]  $ such that the Laplace-Fourier
series (cf. (\ref{gauss}))
\[
f\left(  x\right)  =\sum_{k=0}^{\deg f}\sum_{l=1}^{a_{k}}p_{k,l}(\left|
x\right|  ^{2})Y_{k,l}\left(  x\right)
\]
satisfies the restriction
\[
\deg p_{k,l}\leq n\text{ for }k=0,...,k_{0}\text{ and }p_{k,l}=0\text{ for all
}k\in\mathbb{N}_{0}\text{ with }k>k_{0}.
\]
A functional $T_{n}:U_{2n}\left(  k_{0}\right)  \rightarrow\mathbb{C}$ is
called \emph{pseudo-positive definite with respect to the orthonormal basis}
$Y_{k,l},l=1,...,a_{k}$, $k\in\mathbb{N}_{0},$ $k\leq k_{0},$ if the component
functionals $T_{n,k,l}:\mathcal{P}_{\leq2n}\rightarrow\mathbb{C}$ defined by
\[
T_{n,k,l}\left(  p\right)  :=T_{n}\left(  p\left(  \left|  x\right|
^{2}\right)  Y_{k,l}\left(  x\right)  \right)  \text{ for }p\in\mathcal{P}%
_{\leq2n}%
\]
satisfy
\begin{align}
T_{n,k,l}\left(  p^{\ast}p\right)   &  \geq0\text{ for all }p\in
\mathcal{P}_{\leq n},\label{eq1}\\
T_{n,k,l}\left(  t\cdot p^{\ast}\left(  t\right)  p\left(  t\right)  \right)
&  \geq0\text{ for all }p\in\mathcal{P}_{\leq n-1}. \label{eq2}%
\end{align}
If $k_{0}<\infty$, the space $U_{n}\left(  k_{0}\right)  $ is obviously
finite-dimensional and in this case we can solve the truncated moment problem:

\begin{theorem}
\label{ThmTT}Suppose that $n$ and $k_{0}$ are natural numbers. If
$T_{n}:U_{2n}\left(  k_{0}\right)  \rightarrow\mathbb{C}$ is pseudo-positive
definite with respect to the orthonormal basis $Y_{k,l},l=1,...,a_{k}$,
$k\in\mathbb{N}_{0}$ then there exists a pseudo-positive measure $\sigma$ such
that
\[
T_{n}\left(  P\right)  =\int P\left(  x\right)  d\sigma\left(  x\right)
\]
for all $P\in U_{2n-1}\left(  k_{0}\right)  .$
\end{theorem}

\begin{proof}
Let $k\in\left\{  0,...,k_{0}\right\}  $ and let $T_{n,k,l}:\mathcal{P}%
_{\leq2n}\rightarrow\mathbb{C}$ be the component functional. In the first case
assume that there exists a polynomial $p_{m}\in\mathcal{P}_{\leq n},p_{m}%
\neq0$ with $T_{n,k,l}\left(  p_{m}^{\ast}p_{m}\right)  =0.$ We may assume
that $p_{m}$ has minimal degree, say $m\leq n.$ Then $T_{n,k,l}\left(
p^{\ast}p\right)  >0$ for all $p\in\mathcal{P}_{\leq m-1},p\neq0$. Using the
Gauss-Jacobi quadrature for the functional $T_{n,k,l}$ restricted to
$\mathcal{P}_{\leq2m}$ it follows that there exist points $t_{1,k,l}%
<...<t_{m,k,l}\in\mathbb{R}$ and weights $\alpha_{1,k,l},...,\alpha_{m,k,l}>0$
such that the measure $\sigma_{k,l}:=\alpha_{1,k,l}\delta_{t_{1,k,l}%
}+...+\alpha_{m,k,l}\delta_{t_{m,k,l}}$ coincides with $T_{n,k,l}$ on
$\mathcal{P}_{\leq2m-1}.$ Moreover, condition (\ref{eq2}) implies that
$t_{1,k,l}>0.$ By the Cauchy-Schwarz inequality we have for all $q\in
\mathcal{P}_{\leq2n-m}$
\[
\left|  T_{n,k,l}\left(  q\cdot p_{m}\left(  t\right)  \right)  \right|
^{2}\leq T_{n,k,l}\left(  q^{\ast}q\right)  T_{n,k,l}\left(  p_{m}^{\ast}%
p_{m}\right)  =0.
\]
It follows that $T_{n,k,l}$ and $\sigma_{k,l}$ coincide on $\mathcal{P}%
_{\leq2n-1}.$ Hence we have proved that there exists a non-negative moment
measure $\sigma_{k,l}$ with support in $\left[  0,\infty\right)  $ such that
$T_{n,k,l}\left(  p\right)  =\int_{0}^{\infty}p\left(  t\right)  d\sigma
_{k,l}\left(  t\right)  $ for all $p\in\mathcal{P}_{\leq2n-1},$ and (since
$t_{1,k,l}>0$ )
\begin{equation}
\int_{0}^{\infty}r^{-k}d\sigma_{k,l}<\infty. \label{eqfinite}%
\end{equation}
In the second case, we have $T_{n,k,l}\left(  p^{\ast}p\right)  >0$ for all
$p\in\mathcal{P}_{\leq n},p\neq0$. Using the Gauss-Jacobi quadrature again one
obtains a non-negative moment measure $\sigma_{k,l}$ with support in $\left[
0,\infty\right)  $ such that $T_{n,k,l}\left(  p\right)  =\int_{0}^{\infty
}p\left(  t\right)  d\sigma_{k,l}\left(  t\right)  $ for all $p\in
\mathcal{P}_{\leq2n-1},$ satisfying (\ref{eqfinite}).

Let $\sigma_{k,l}$ for $k=0,...,k_{0}$ be as above and define $\sigma_{k,l}=0$
for $k>k_{0}$. Define a functional $T:\mathbb{C}\left[  x_{1},...,x_{d}%
\right]  \rightarrow\mathbb{C}$ by
\[
T\left(  f\right)  :=\sum_{k=0}^{\deg f}\sum_{l=1}^{a_{k}}\int_{0}^{\infty
}f_{k,l}\left(  r\right)  r^{-k}d\sigma_{k,l}.
\]
By Theorem \ref{ThmRepG} (note that the summability condition is satisfied)
there exists a pseudo-positive moment measure $\sigma$ with the same moments
as $T.$ The proof is accomplished by the fact that $T$ and $T_{n}$ agree on
the subspace $U_{2n-1}\left(  k_{0}\right)  .$
\end{proof}

Now we consider the case $k_{0}=\infty$, so the space $U_{n}\left(
k_{0}\right)  $ is infinite-dimensional. Using the same method of proof one obtains:

\begin{theorem}
\label{TrepTruncated}Suppose that $n$ is a natural number and that
$T_{n}:U_{2n}\left(  \infty\right)  \rightarrow\mathbb{C}$ is pseudo-positive
definite with respect to the orthonormal basis $Y_{k,l},l=1,...,a_{k}$,
$k\in\mathbb{N}_{0}.$ Assume that the non-negative measures $\sigma_{k,l}$
constructed in the proof of Theorem \ref{ThmTT} satisfy the following
conditions
\[
C_{N}:=\sum_{k=0}^{\infty}\sum_{l=1}^{a_{k}}\int_{0}^{\infty}r^{N}%
r^{-k}d\sigma_{k,l}<\infty
\]
for any $N\in\mathbb{N}_{0}.$ Then there exists a pseudo-positive, signed
moment measure $\sigma$ such that
\[
T\left(  f\right)  =\int_{\mathbb{R}^{n}}fd\sigma\text{ for all }f\in
U_{2n-1}\left(  \infty\right)  .
\]
\end{theorem}

\begin{remark}
Let us note that (in the case $k_{0}=\infty$ ) the space $U_{2n}\left(
\infty\right)  $ coincides with the set of all polynomials $h$ which are
polyharmonic of order $n+1$, i.e. satisfy $\Delta^{n+1}h=0$ , where $\Delta$
is the Laplace operator and $\Delta^{j}$ is its $j$th iterate. Apparently for
the first time such representing measures have been considered more
systematically in \cite{schulzewildenhain}. In the case $n=0$ the problem we
consider is equivalent to the inverse magnetic problem, cf. \cite{zidarov}.
\end{remark}

\section{Determinacy for pseudo-positive definite
functionals\label{Sdeterminacy}}

Let $M^{\ast}\left(  \mathbb{R}^{d}\right)  $ be the set of all \emph{signed
moment measures}, and $M_{+}^{\ast}\left(  \mathbb{R}^{d}\right)  $ be the set
of \emph{non--negative moment measures} on $\mathbb{R}^{d}$. On $M^{\ast
}\left(  \mathbb{R}^{d}\right)  $ we define an equivalence relation: we say
that $\sigma\sim\mu$ \ for two elements $\sigma,\mu\in M^{\ast}\left(
\mathbb{R}^{d}\right)  $ if and only if $\int_{\mathbb{R}^{d}}fd\sigma
=\int_{\mathbb{R}^{d}}fd\mu$ for all $f\in\mathbb{C}\left[  x_{1}%
,...,x_{d}\right]  .$

\begin{definition}
\label{Ddetermined}Let $\mu\in$ $M^{\ast}\left(  \mathbb{R}^{d}\right)  $ be a
pseudo-positive measure. We define
\[
V_{\mu}=\left\{  \sigma\in M^{\ast}\left(  \mathbb{R}^{d}\right)
:\sigma\text{ is pseudo-positive and }\sigma\sim\mu\right\}  .
\]
We say that the measure $\mu\in$ $M^{\ast}\left(  \mathbb{R}^{d}\right)  $ is
\emph{determined in the class of pseudo-positive measures} if $V_{\mu}$ has
only one element, i.e. is equal to $\left\{  \mu\right\}  .$
\end{definition}

Recall that a positive definite functional $\phi:\mathcal{P}_{1}%
\rightarrow\mathbb{R}$ is \emph{determined in the sense of Stieltjes} if the
set
\begin{equation}
W_{\phi}^{Sti}:=\left\{  \tau\in M_{+}^{\ast}\left(  \left[  0,\infty\right)
\right)  :\int_{0}^{\infty}r^{m}d\tau=\phi\left(  r^{m}\right)  \text{ for all
}m\in\mathbb{N}_{0}\right\}  \label{defES}%
\end{equation}
has exactly one element, cf. \cite[p. 210]{BeTh91}.

According to Proposition \ref{pseudopos}, we can associate to a
pseudo-positive measure $\mu$ the sequence of non-negative measures $\mu
_{k,l},k\in\mathbb{N}_{0},l=1,..,a_{k}$ with support in $\left[
0,\infty\right)  .$ The measures $\mu_{k,l}$ contain all information about
$\mu.$ Indeed, we prove

\begin{proposition}
\label{soso}Let $\mu$ and $\sigma$ be pseudo-positive measures and let
$\mu_{k,l}$ and $\sigma_{k,l}$ be as in Proposition \ref{pseudopos}. If
$\mu_{k,l}=\sigma_{k,l}$ for all $k\in\mathbb{N}_{0},l=1,..,a_{k}$ then
$\mu=\sigma.$
\end{proposition}

\begin{proof}
Let $h\in C_{c}\left[  0,\infty\right)  .$ Then, using the assumption
$\mu_{k,l}=\sigma_{k,l},$ we obtain
\[
\int_{\mathbb{R}^{d}}h\left(  \left|  x\right|  \right)  Y_{k,l}\left(
x\right)  d\mu=\int_{0}^{\infty}h\left(  t\right)  d\mu_{k,l}=\int
_{\mathbb{R}^{d}}h\left(  \left|  x\right|  \right)  Y_{k,l}\left(  x\right)
d\sigma.
\]
Since each $f\in C^{\times}\left(  \mathbb{R}^{d}\right)  \cap C_{c}\left(
\mathbb{R}^{d}\right)  $ is a finite linear combination of functions of the
type $h\left(  \left|  x\right|  \right)  Y_{k,l}\left(  x\right)  $ with
$h\in C_{c}\left[  0,\infty\right)  $, we obtain that $\int_{\mathbb{R}^{d}%
}fd\mu=\int_{\mathbb{R}^{d}}fd\sigma$ for all $f\in C^{\times}\left(
\mathbb{R}^{d}\right)  \cap C_{c}\left(  \mathbb{R}^{d}\right)  .$ We apply
Proposition \ref{Ldensity} to see that $\mu$ is equal to $\sigma.$
\end{proof}

The following result is proved in \cite[Proposition 3.1]{BeTh91}:

\begin{proposition}
\label{Ldensity} Let $\mu$ and $\sigma$ be signed measures on $\mathbb{R}%
^{d}.$ If $\int_{\mathbb{R}^{d}}fd\mu=\int_{\mathbb{R}^{d}}fd\sigma$ for all
$f\in C^{\times}\left(  \mathbb{R}^{d}\right)  \cap C_{c}\left(
\mathbb{R}^{d}\right)  ,$ then $\mu$ is equal to $\sigma.$
\end{proposition}

We can characterize $V_{\mu}$ in the case that only finitely many $\mu_{k,l}$
are nonzero.

\begin{theorem}
Let $\mu$ be a pseudo-positive measure on $\mathbb{R}^{n}$ such that
$\mu_{k,l}=0$ for all $k>k_{0},l=1,...,a_{k}.$ Then $V_{\mu}$ is affinely
isomorphic to the set
\begin{equation}
\oplus_{k=0}^{k_{0}}\oplus_{l=1}^{a_{k}}\{\rho_{k,l}\in W_{\mu_{k,l}^{\psi}%
}^{Sti}:\int_{0}^{\infty}t^{-\frac{1}{2}k}d\rho_{k,l}<\infty\}
\label{descript}%
\end{equation}
where the isomorphism is given by $\sigma\longmapsto\left(  \sigma_{k,l}%
^{\psi}\right)  _{k=1,..,k_{0},l=1,...,a_{k}}$ and the map $\psi:\left[
0,\infty\right)  \rightarrow\left[  0,\infty\right)  $ is defined by
$\psi\left(  t\right)  =t^{2},$ cf. (\ref{imagemeasure1}).
\end{theorem}

\begin{proof}
Let $\sigma$ be in $V_{\mu}.$ Let $\sigma_{k,l}$ and $\mu_{k,l}$ be the unique
moment measures obtained in Proposition \ref{pseudopos}. Then
\[
\int_{0}^{\infty}h\left(  t\right)  d\sigma_{k,l}^{\psi}=\int_{0}^{\infty
}h\left(  t^{2}\right)  d\sigma_{k,l}=\int_{\mathbb{R}^{n}}h(\left|  x\right|
^{2})Y_{k,l}\left(  x\right)  d\sigma\left(  x\right)
\]
for all $h\in C_{pol}\left[  0,\infty\right)  ,$ and an analog equation is
valid for $\mu_{k,l}$ and $\mu.$ Taking polynomials $h\left(  t\right)  $ we
see that $\sigma_{k,l}\in W_{\mu_{k,l}^{\psi}}^{Sti}$ using the assumption
that $\mu\sim\sigma.$ Using a simple approximation argument it is easy to see
from (\ref{eqlim}) that
\[
\int_{0}^{\infty}t^{-\frac{1}{2}k}d\sigma_{k,l}^{\psi}=\int_{\mathbb{R}^{n}%
}Y_{k,l}\left(  \frac{x}{\left|  x\right|  }\right)  d\sigma\left(  x\right)
.
\]
Since $x\longmapsto Y_{k,l}\left(  \frac{x}{\left|  x\right|  }\right)  $ is
bounded on $\mathbb{R}^{n}$, say by $M,$ we obtain the estimate
\[
\left|  \int_{0}^{\infty}t^{-\frac{1}{2}k}d\sigma_{k,l}^{\psi}\right|  \leq
M\int_{\mathbb{R}^{n}}1d\left|  \sigma\right|  <\infty.
\]
It follows that $\left(  \sigma_{k,l}^{\psi}\right)  _{k=1,..,k_{0}%
,l=1,...,a_{k}}$ is contained in the set on the right hand side in
(\ref{descript}).

Let now $\rho_{k,l}\in W_{\mu_{k,l}^{\psi}}^{Sti}$ be given such that
$\int_{0}^{\infty}t^{-\frac{1}{2}k}d\rho_{k,l}<\infty$ for $k=1,..,k_{0}%
,l=1,...,a_{k}.$ Define $\sigma_{k,l}=\rho_{k,l}^{\psi^{-1}}$ and
$\sigma_{k,l}=0$ for $k>k_{0}$. Then by Theorem \ref{ThmRepG} there exists a
measure $\tau\in V_{\mu}$ such that $\tau_{k,l}=\sigma_{k,l}.$ This shows the
surjectivity of the map. Let now $\sigma$ and $\tau$ are in $V_{\mu}$ with
$\sigma_{k,l}^{\psi}=\tau_{k,l}^{\psi}$ for $k=1,..,k_{0},l=1,...,a_{k}.$ The
property $\sigma\in V_{\mu}$ implies that $\sigma_{k,l}^{\psi}\in W_{\mu
_{k,l}^{\psi}}^{Sti}$ for all $k\in\mathbb{N}_{0},l=1,...,a_{k},$ hence
$\sigma_{k,l}^{\psi}=0$ for $k>k_{0},$ and similarly $\tau_{k,l}^{\psi}=0.$
Hence $\sigma_{k,l}=\tau_{k,l}$ for all $k\in\mathbb{N}_{0},l=1,...,a_{k},$
and this implies that $\sigma=\tau$ by Proposition \ref{soso}.
\end{proof}

The following is a sufficient condition for a functional $T$ to be determined
in the class of pseudo-positive measures.

\begin{theorem}
\label{ThmUniq}Let $T:\mathbb{C}\left[  x_{1},...,x_{d}\right]  \rightarrow
\mathbb{R}$ be a pseudo-positive definite functional. If the functionals
$T_{k,l}:\mathbb{C}\left[  x_{1}\right]  \rightarrow\mathbb{C}$ are determined
in the sense of Stieltjes then there exists at most one pseudo-positive,
signed moment measure $\mu$ on $\mathbb{R}^{d}$ with
\begin{equation}
T\left(  f\right)  =\int_{\mathbb{R}^{d}}fd\mu\qquad\text{ for all }%
f\in\mathbb{C}\left[  x_{1},...,x_{d}\right]  . \label{eqTident}%
\end{equation}
\end{theorem}

\begin{proof}
Let us suppose that $\mu$ and $\sigma$ are pseudo-positive, signed moment
measures on $\mathbb{R}^{d}$ representing $T.$ Taking $f=\left|  x\right|
^{2N}Y_{k,l}\left(  x\right)  $ we obtain from (\ref{eqTident}) that
\[
\int_{\mathbb{R}^{d}}\left|  x\right|  ^{2N}Y_{k,l}\left(  x\right)
d\mu=T_{k,l}\left(  t^{N}\right)  =\int_{\mathbb{R}^{d}}\left|  x\right|
^{2N}Y_{k,l}\left(  x\right)  d\sigma.
\]
for all $N\in\mathbb{N}_{0}.$ Let $\mu_{k,l}$ and $\sigma_{k,l}$ as in
Proposition \ref{pseudopos}, and consider $\psi:\left[  0,\infty\right)
\rightarrow\left[  0,\infty\right)  $ defined by $\psi\left(  t\right)
=t^{2}$. Then the image measures $\mu_{k,l}^{\psi}$ and $\sigma_{k,l}^{\psi}$
are non-negative measures with supports on $\left[  0,\infty\right)  $ such
that $\int_{0}^{\infty}t^{N}d\mu_{k,l}^{\psi}=T_{k,l}\left(  t^{N}\right)
=\int_{0}^{\infty}t^{N}d\sigma_{k,l}^{\psi}.$ Our assumption implies that
$\mu_{k,l}^{\psi}=\sigma_{k,l}^{\psi}$, so $\mu_{k,l}=\sigma_{k,l}$.
Proposition \ref{soso} implies that $\mu$ is equal to $\sigma$.
\end{proof}

In the following we want to prove the converse of the last theorem, which is
more subtle. We need now some special results about \emph{Nevanlinna extremal
measures}. Let us introduce the following notation: for a non-negative measure
$\phi\in M_{+}^{\ast}\left(  \mathbb{R}\right)  $ we put\footnote{Here in
order to avoid mixing of the notations, we retain the notation $\left[
\phi\right]  $ from the one--dimensional case in \cite{BeTh91}.}
\[
\left[  \phi\right]  :=\left\{  \sigma\in M_{+}^{\ast}\left(  \mathbb{R}%
\right)  :\text{ }\sigma\sim\phi\right\}  .
\]

\begin{proposition}
\label{PropSD} Let $\nu$ be a non-negative moment measure on $\mathbb{R}$ with
support in $\left[  0,\infty\right)  $ which is not determined in the sense of
Stieltjes, or applying the notation (\ref{defES}) $W_{\nu}^{Sti}\neq\left\{
\nu\right\}  .$ Then there exist uncountably many $\sigma\in W_{\nu}^{Sti}$
such that $\int_{0}^{\infty}u^{-k}d\sigma<\infty$ for all $k\in\mathbb{N}_{0}.$
\end{proposition}%

%TCIMACRO{\TeXButton{Proof}{\proof}}%
%BeginExpansion
\proof
%EndExpansion
In the proof we will borrow some arguments about the Stieltjes problem as
given in \cite{Chih82} or \cite{Pede95}. As in the proof of Proposition 4.1 in
\cite{Pede95} let $\varphi:\left(  -\infty,\infty\right)  \rightarrow\left[
0,\infty\right)  $ be defined by $\varphi\left(  x\right)  =x^{2}.$ If
$\lambda$ is a measure on $\mathbb{R}$ define a measure $\lambda^{-}$ by
$\lambda^{-}\left(  A\right)  :=\lambda\left(  -A\right)  $ for each Borel set
$A$ where $-A:=\left\{  -x:x\in A\right\}  .$ The measure is symmetric if
$\lambda^{-}=\lambda.$ For each $\tau\in W_{\nu}^{Sti}$ define a measure
$\widetilde{\tau}:=\frac{1}{2}\left(  \tau^{\varphi}+\left(  \tau^{\varphi
}\right)  ^{-}\right)  $ which is clearly symmetric, in particular
$\widetilde{\nu}$ is symmetric. As pointed out in \cite{Pede95}, the map
$\widetilde{\cdot}:W_{\nu}^{Sti}\rightarrow\left[  \widetilde{\nu}\right]  \ $
is injective and the image is exactly the set of all symmetric measures in the
set $\left[  \widetilde{\nu}\right]  .$ The inverse map of $\widetilde{\cdot}$
defined on the image space is just the map $\sigma\rightarrow\sigma^{\varphi}.$

It follows that $\widetilde{v}$ is not determined, so we can make use of the
Nevanlinna theory for the indeterminate measure $\widetilde{\nu},$ see p. 54
in \cite{Akhi65}. We know by formula II.4.2 (9) and II.4.2 (10) in
\cite{Akhi65} that for every $t\in\mathbb{R}$ there exists a unique
Nevanlinna--extremal measure $\sigma_{t}$ such that
\[
\int_{-\infty}^{\infty}\frac{d\sigma_{t}\left(  u\right)  }{u-z}%
=-\frac{A\left(  z\right)  t-C\left(  z\right)  }{B\left(  z\right)
t-D\left(  z\right)  },
\]
where $A\left(  z\right)  ,B\left(  z\right)  ,C\left(  z\right)  ,D\left(
z\right)  $ are entire functions. Since the support of $\sigma_{t}$ is the
zero-set of the entire function $B\left(  z\right)  t-D\left(  z\right)  $ it
follows that the measure $\sigma_{t}$ has no mass in $0$ for $t\neq0,$ and now
it is clear that $\sigma_{t}([-\delta,\delta])=0$ for $t\neq0$ and suitable
$\delta>0$ (this fact is pointed out at least in the reference \cite[p.
210]{BeTh91}). It follows that
\begin{equation}
\int_{-\infty}^{\infty}\left|  u\right|  ^{-k}d\sigma_{t}<\infty
\label{eqbounded}%
\end{equation}
since the function $u\longmapsto\left|  u\right|  ^{-k}$ is bounded on
$\mathbb{R}\setminus\left[  -\delta,\delta\right]  $ for each $\delta>0$.
Using the fact that the functions $A\left(  z\right)  $ and $B\left(
z\right)  $ of the Nevanlinna matrix are odd, while the functions $B\left(
z\right)  $ and $C\left(  z\right)  $ are even, one derives that the measure
$\rho_{t}:=\frac{1}{2}\sigma_{t}+\frac{1}{2}\sigma_{-t}$ is symmetric. Further
from the equation $A\left(  z\right)  D\left(  z\right)  -B\left(  z\right)
C\left(  z\right)  =1$ it follows that $\rho_{t}\neq\rho_{s}$ for positive
numbers $t\neq s.$ By the above we know that $\rho_{t}^{\varphi}\neq\rho
_{s}^{\varphi}$. This finishes the proof.
%TCIMACRO{\TeXButton{End Proof}{\endproof}}%
%BeginExpansion
\endproof
%EndExpansion

\begin{theorem}
Let $\mu$ be a pseudo-positive signed measure on $\mathbb{R}^{d}$ such that
the summability assumption (\ref{maincond2}) holds. Then $V_{\mu}$ contains
exactly one element if and only if each $\mu_{k,l}^{\psi}$ is determined in
the sense of Stieltjes.
\end{theorem}

\begin{proof}
Let $\mu_{k,l}$ be the component measures as defined in Proposition
\ref{pseudopos}. Assume that $V_{\mu}=\left\{  \mu\right\}  $ but that some
$\tau:=\mu_{k_{0},l_{0}}^{\psi}$ is not determined in the sense of Stieltjes
where $\psi\left(  t\right)  =t^{2}$ for $t\in\left[  0,\infty\right)  .$ By
Proposition \ref{PropSD} there exists a measure $\sigma\in W_{\tau}^{Sti}$
such $\sigma\neq\tau$ and $\int_{0}^{\infty}r^{-k}d\sigma<\infty.$ By Theorem
\ref{ThmRepG} there exists a pseudo-positive moment measure $\widetilde{\mu}$
representing the functional
\[
\widetilde{T}\left(  f\right)  :=\sum_{k=0,k\neq k_{0}}^{\infty}%
\sum_{l=1,l\neq l_{0}}^{a_{k}}\int_{0}^{\infty}f_{k,l}\left(  r\right)
r^{-k}d\mu_{k,l}+\int_{0}^{\infty}f_{k_{0},l_{0}}\left(  r\right)
r^{-k}d\sigma^{\psi^{-1}}.
\]
Then $\widetilde{\mu}$ is different from $\mu$ since $\sigma^{\psi^{-1}}%
\neq\mu_{k_{0},l_{0}}$ and $\widetilde{\mu}\in V_{\mu}$ since $\sigma\in
W_{\tau}^{Sti}.$ This contradiction shows that $\mu_{k_{0},l_{0}}^{\psi}$ is
determined in the sense of Stieltjes. The sufficiency follows from Theorem
\ref{ThmUniq}. The proof is complete.
\end{proof}

\section{Miscellaneous results\label{Sexamples}}

In this section we provide some examples and results on pseudo--positive
measures which throw more light on these new notions.

\subsection{The univariate case\label{Sunivariate}}

As we mentioned in the Introduction the non-negative spherically symmetric
measures are pseudopositive and as it is easy to see from (\ref{cjkl}) our
theory reduces to the classical Stieltjes moment problem. Other pseudopositive
measures $\mu$ for which our theory reduces essentially to the Stieltjes
one-dimensional moment problem  are those having only one component measure
$\mu_{k,l}$ non-zero; this is the problem $\int_{0}^{\infty}r^{k+2j}d\mu
_{k,l}\left(  r\right)  =c_{j,k,l}$ for  $j=0,1,2,...$,  (cf. (by
(\ref{eqlim}) and  (\ref{cjkl})). 

On the other hand it is instructive to consider the univariate case of our
theory: then $d=1,$ $\mathbb{S}^{0}=\left\{  -1,1\right\}  ,$ and the
normalized measure is $\omega_{0}\left(  \theta\right)  =\frac{1}{2}$ for all
$\theta\in\mathbb{S}^{0}.$ The harmonic polynomials are the linear functions,
their basis are the two functions defined by $Y_{0}\left(  x\right)  =1$ and
$Y_{1}\left(  x\right)  =x$ for all $x\in\mathbb{R}.$ The following is now
immediate from the definitions:

\begin{proposition}
Let $d=1.$ A functional $T:\mathbb{C}\left[  x\right]  \rightarrow\mathbb{C\ }
$ is pseudo-positive definite if and only if $T\left(  p^{\ast}\left(
x^{2}\right)  p\left(  x^{2}\right)  \right)  \geq0$ and $T\left(  xp^{\ast
}\left(  x^{2}\right)  p\left(  x^{2}\right)  \right)  \geq0$ for all
$p\in\mathbb{C}\left[  x\right]  .$
\end{proposition}

Recall that a functional $T:\mathbb{C}\left[  x\right]  \rightarrow
\mathbb{C\ }$defines a \emph{Stieltjes moment sequence} if $T\left(  q^{\ast
}\left(  x\right)  q\left(  x\right)  \right)  \geq0$ and $T\left(  xq^{\ast
}\left(  x\right)  q\left(  x\right)  \right)  \geq0$ for all $q\in
\mathbb{C}\left[  x\right]  ,$ so this property implies pseudo-positive
definiteness; the next example shows that the converse is not true:

\begin{example}
Let $\sigma$ be a non-negative finite measure on the interval $\left[
a,b\right]  $ with $a>0.$ Then the functional $T:\mathbb{C}\left[  x\right]
\rightarrow\mathbb{C\ }$ defined by
\[
T\left(  f\right)  =\int_{a}^{b}f\left(  x\right)  d\sigma-\int_{a}%
^{b}f\left(  -x\right)  d\sigma
\]
is pseudo-positive definite but not positive definite.
\end{example}

\begin{proof}
As pointed out in \cite[Chapter 4.1]{StWe71}, the Laplace--Fourier expansion
of $f$ is given by $f\left(  r\theta\right)  =f_{0}\left(  r\right)
Y_{0}\left(  \theta\right)  +f_{1}\left(  r\right)  Y_{1}\left(
\theta\right)  $ for $x=r\theta$ with $r=\left|  x\right|  $ and $\theta
\in\mathbb{S}^{0},$ where
\begin{align*}
f_{0}\left(  r\right)   &  =\int_{\mathbb{S}^{0}}f\left(  r\theta\right)
Y_{0}\left(  \theta\right)  d\omega_{0}\left(  \theta\right)  =\frac{f\left(
r\right)  +f\left(  -r\right)  }{2},\\
f_{1}\left(  r\right)   &  =\int_{\mathbb{S}^{0}}f\left(  r\theta\right)
Y_{1}\left(  \theta\right)  d\omega_{0}\left(  \theta\right)  =\frac{f\left(
r\right)  -f\left(  -r\right)  }{2}.
\end{align*}
Since $f_{0}$ is even and $f_{1}$ is odd and $f=f_{0}+f_{1}$ we infer that
$T\left(  f\right)  =2\int_{a}^{b}f_{1}\left(  r\right)  d\sigma$. By
Proposition \ref{PropTTT} $T$ is pseudo-positive definite. Since $T\left(
1\right)  =0$ and $T\neq0$ it is clear that $T$ is not positive definite.
\end{proof}

\subsection{A criterion for pseudo-positivity}

The following is a simple criterion for pseudo-positivity:

\begin{proposition}
\label{ThmLaplace}Let $\mu$ be a signed moment measure on $\mathbb{R}^{d}.$
Assume that $\mu$ has a density $w\left(  x\right)  $ with respect to the
Lebesgue measure $dx$ such that $\theta\longmapsto w\left(  r\theta\right)  $
is in $L^{2}\left(  \mathbb{S}^{d-1}\right)  $ for each $r>0.$ If the
Laplace-Fourier coefficients of $w,$
\[
w_{k,l}\left(  r\right)  :=\int_{\mathbb{S}^{d-1}}w\left(  r\theta\right)
Y_{k,l}\left(  \theta\right)  d\theta
\]
are non-negative then $\mu$ is pseudo-positive and
\begin{align}
d\mu_{k,l}\left(  r\right)   &  =r^{k+d-1}w_{k,l}\left(  r\right)
,\label{powermmuk}\\
\int_{0}^{\infty}r^{-k}d\mu_{k,l}\left(  r\right)   &  =\int_{0}^{\infty
}w_{k,l}\left(  r\right)  \cdot r^{d-1}dr \label{powermmuk2}%
\end{align}
if the last integral exists. The measures $\mu_{k,l}$ are defined by means of
equality (\ref{eqlim}).
\end{proposition}

\begin{proof}
Since $\mu$ has a density $w\left(  x\right)  $ we can use polar coordinates
to obtain for $f\in C_{pol}\left(  \mathbb{R}^{d}\right)  $
\begin{equation}
\int_{\mathbb{R}^{d}}fd\mu=\int_{\mathbb{R}^{d}}f\left(  x\right)  w\left(
x\right)  dx=\int_{0}^{\infty}\int_{\mathbb{S}^{d-1}}f\left(  r\theta\right)
w\left(  r\theta\right)  r^{d-1}d\theta dr. \label{eqPolar}%
\end{equation}
For any $h\in C_{pol}\left[  0,\infty\right)  $ we put $f\left(  x\right)
=h\left(  \left|  x\right|  \right)  Y_{k,l}\left(  x\right)  ,$ then we
obtain
\begin{equation}
\int_{\mathbb{R}^{d}}h\left(  \left|  x\right|  \right)  Y_{k,l}\left(
x\right)  d\mu=\int_{0}^{\infty}\int_{\mathbb{S}^{d-1}}h\left(  r\right)
r^{k+d-1}Y_{k,l}\left(  \theta\right)  w\left(  r\theta\right)  d\theta dr.
\label{eqPolar2}%
\end{equation}
Since $\theta\longmapsto w\left(  r\theta\right)  $ is in $L^{2}\left(
\mathbb{S}^{d-1}\right)  $, we know that $w_{k,l}\left(  r\right)
=\int_{\mathbb{S}^{d-1}}w\left(  r\theta\right)  Y_{k,l}\left(  \theta\right)
d\theta.$ Hence, by the definition of $\mu_{k,l},$ we obtain
\begin{equation}
\int_{0}^{\infty}h\left(  r\right)  d\mu_{k,l}:=\int_{\mathbb{R}^{d}}h\left(
\left|  x\right|  \right)  Y_{k,l}\left(  x\right)  d\mu=\int_{0}^{\infty
}h\left(  r\right)  w_{k,l}\left(  r\right)  r^{k+d-1}dr. \label{eqLaplacemu}%
\end{equation}
Thus the measure $\mu$ is pseudo-positive, and (\ref{powermmuk}) follows. Let
us prove (\ref{powermmuk2}): we define the cut--off functions $h_{m}\in
C_{pol}\left[  0,\infty\right)  $ such that $h_{m}\left(  t\right)  =t^{-k}$
for $t\geq1/m$ and such that $h_{m}\leq h_{m+1}.$ Now use (\ref{eqLaplacemu})
and the monotone convergence theorem to obtain (\ref{powermmuk2}).
\end{proof}

\subsection{Examples in the two--dimensional case\label{S2d}}

Let us consider the case $d=2,$ and take the usual orthonormal basis of solid
harmonics, defined by $Y_{0}\left(  e^{it}\right)  =\frac{1}{2\pi}$ and
\begin{equation}
Y_{k,1}\left(  re^{it}\right)  =\frac{1}{\sqrt{\pi}}r^{k}\cos kt\text{ and
}Y_{k,2}\left(  re^{it}\right)  =\frac{1}{\sqrt{\pi}}r^{k}\sin kt\text{ for
}k\in\mathbb{N}. \label{eqY}%
\end{equation}
We define a density $w^{\left(  \alpha\right)  }:\mathbb{R}^{n}\rightarrow
\left[  0,\infty\right)  $, depending on parameter $\alpha>0$, by
\begin{align*}
w^{\left(  \alpha\right)  }\left(  re^{it}\right)   &  :=\left(  1-r^{\alpha
}\right)  P\left(  re^{it}\right)  \qquad\text{ for }0\leq r<1\\
w^{\left(  \alpha\right)  }\left(  re^{it}\right)   &  =0\qquad\text{for
}r\geq1;
\end{align*}
here the function $P\left(  re^{it}\right)  $ is the Poisson kernel for $0\leq
r<1$ given by (see e.g. 5.1.16 in \cite[p. 243]{AAR99})
\begin{equation}
P\left(  re^{it}\right)  :=\frac{1-r^{2}}{1-2r\cos t+r^{2}}=1+\sum
_{k=1}^{\infty}2r^{k}\cos kt. \label{eqPoisson}%
\end{equation}
By Proposition \ref{ThmLaplace}, the measure $d\mu^{\alpha}:=w^{\left(
\alpha\right)  }\left(  x\right)  dx$ is pseudo-positive. For $k>0,$ by
(\ref{powermmuk2}) and (\ref{eqY}) we obtain
\[
\int r^{-k}d\mu_{k,1}^{\alpha}=2\sqrt{\pi}\int_{0}^{1}r^{k+1}\left(
1-r^{\alpha}\right)  dr=\frac{2\sqrt{\pi}\alpha}{\left(  k+2\right)  \left(
\alpha+k+2\right)  }.
\]
It follows that $w^{\left(  \alpha\right)  }\left(  x\right)  dx$ satisfies
the summability condition (\ref{maincond2}).

On the other hand, there exist pseudo-positive measures which do not satisfy
the summability condition (\ref{maincond2}):

\begin{proposition}
Let $w\left(  re^{it}\right)  :=P\left(  re^{it}\right)  $ for $0\leq r<1$ and
$w\left(  re^{it}\right)  :=0$ for $r\geq1$ where $P\left(  x\right)  $ is
given by (\ref{eqPoisson}). Then $d\mu:=w\left(  x\right)  dx$ is a
pseudo-positive, non-negative moment measure which does not satisfy the
summability condition (\ref{maincond2}).
\end{proposition}

\begin{proof}
It follows from (\ref{powermmuk2}) for $k\geq1$
\[
\int r^{-k}d\mu_{k,1}=\int_{0}^{\infty}w_{k,1}\left(  r\right)  \cdot
r^{d-1}dr=2\sqrt{\pi}\int_{0}^{1}r^{k+1}dr=\frac{2\sqrt{\pi}}{\left(
k+2\right)  },
\]
so we see that the summability condition (\ref{maincond2}) is not fulfilled.
\end{proof}

\subsection{The summability condition}

The next result shows that the spectrum of the measures $\sigma_{k,l}$ is
contained in the spectrum of the representation measure $\mu$.

\begin{theorem}
Let $\sigma_{k,l}$ be non-negative measures on $\left[  0,\infty\right)  $. If
the functional $T:\mathbb{C}\left[  x_{1},...,x_{d}\right]  \rightarrow
\mathbb{C}$ defined by (\ref{defTTT}) possesses a representing moment measure
$\mu$ with compact support then
\[
\sigma_{k,l}(\left\{  \left|  x\right|  ^{2}\right\}  )\leq\max_{\theta
\in\mathbb{S}^{d-1}}\left|  Y_{k,l}\left(  \theta\right)  \right|
\cdot\left|  x\right|  ^{k}\cdot\left|  \mu\right|  \left(  \left|  x\right|
^{2}\mathbb{S}^{d-1}\right)
\]
for any $x\in\mathbb{R}^{d}$ where $\left|  \mu\right|  $ is the total
variation and $\left|  x\right|  ^{2}\mathbb{S}^{d-1}=\{\left|  x\right|
^{2}\theta:\theta\in\mathbb{S}^{d-1}\}.$
\end{theorem}

\begin{proof}
Let the support of $\mu$ be contained in $B_{R}.$ Let $x_{0}\in\mathbb{R}^{d}
$ be given. For every univariate polynomial $p\left(  t\right)  $ with
$p\left(  \left|  x_{0}\right|  ^{2}\right)  =1$ we have
\begin{align*}
\sigma_{k,l}\left(  \{\left|  x_{0}\right|  ^{2}\}\right)   &  \leq\int
_{0}^{\infty}p\left(  r^{2}\right)  d\sigma_{k,l}\leq\int_{\mathbb{R}^{d}%
}\left|  p(\left|  x\right|  ^{2})Y_{k,l}\left(  x\right)  \right|  d\left|
\mu\right| \\
&  \leq\max_{\theta\in\mathbb{S}^{d-1}}\left|  Y_{k,l}\left(  \theta\right)
\right|  \int_{\mathbb{R}^{d}}\left|  p(\left|  x\right|  ^{2})\right|
\left|  x\right|  ^{k}d\left|  \mu\right|  .
\end{align*}
Now choose a sequence of polynomials $p_{m}$ with $p_{m}\left(  \left|
x_{0}\right|  ^{2}\right)  =1$ which converges on $\left[  0,R\right]  $ to
the function $f$ defined by $f\left(  \left|  x_{0}\right|  ^{2}\right)  =1$
and $f\left(  t\right)  =0$ for $t\neq\left|  x_{0}\right|  ^{2}.$ Since
$\left|  \mu\right|  $ has support in $B_{R}$ Lebesgue's convergence theorem
shows that
\[
\sigma_{k,l}\left(  \{\left|  x_{0}\right|  ^{2}\}\right)  \leq\max_{\theta
\in\mathbb{S}^{d-1}}\left|  Y_{k,l}\left(  \theta\right)  \right|
\int_{\mathbb{R}^{d}}\left|  f\left(  x\right)  \right|  \left|  x\right|
^{k}d\left|  \mu\right|  .
\]
The last implies our statement.
\end{proof}

The following result shows that the summability condition is sometimes
equivalent to the existence of a pseudo-positive representing measure:

\begin{corollary}
Let $d=2.$ Let $\sigma_{k,l}$ be non-negative measures on $\left[
0,\infty\right)  $ and assume that they have disjoint and at most countable
supports. If the functional $T:\mathbb{C}\left[  x_{1},x_{2}\right]
\rightarrow\mathbb{C}$ defined by (\ref{defTTT}) possesses a representing
moment measure with compact support then
\[
\sum_{k=0}^{\infty}\sum_{l=1}^{a_{k}}\int_{0}^{\infty}r^{-k}d\sigma
_{k,l}\left(  r\right)  <\infty\text{.}%
\]
\end{corollary}

\begin{proof}
Let $\Sigma_{k,l}$ be the support set of $\sigma_{k,l}.$ The last theorem
shows that $\sigma_{k,l}\left(  \left\{  0\right\}  \right)  =0,$ hence
$0\notin\Sigma_{k,l}.$ Moreover it tells us that
\[
\int_{0}^{\infty}r^{-k}d\sigma_{k,l}\left(  r\right)  \leq\max_{\theta
\in\mathbb{S}^{d-1}}\left|  Y_{k,l}\left(  \theta\right)  \right|  \cdot
\sum_{r\in\Sigma_{k,l}}\left|  \mu\right|  \left(  r\mathbb{S}^{d-1}\right)
.
\]
Since $d=2$ we know that $\max_{\theta\in\mathbb{S}^{d-1}}\left|
Y_{k,l}\left(  \theta\right)  \right|  \leq1.$ Hence
\[
\sum_{k=0}^{\infty}\sum_{l=1}^{a_{k}}\int_{0}^{\infty}r^{-k}d\sigma
_{k,l}\left(  r\right)  \leq\sum_{k=0}^{\infty}\sum_{l=1}^{a_{k}}\sum
_{r\in\Sigma_{k,l}}\left|  \mu\right|  \left(  r\mathbb{S}^{d-1}\right)
\leq\left|  \mu\right|  \left(  \mathbb{R}^{d}\right)
\]
where the last inequality follows from the fact that $\Sigma_{k,l}$ are
pairwise disjoint.
\end{proof}

Recall that the converse of the last theorem holds under the additional
assumption that the supports of all $\sigma_{k,l}$ are contained in some
interval $\left[  0,R\right]  .$

\begin{theorem}
\label{Tnonexistence}There exists a functional $\ T:\mathbb{C}\left[
x_{1},...,x_{d}\right]  \rightarrow\mathbb{C}$ which is pseudo-positive
definite but does not possess a pseudo-positive representing measure.
\end{theorem}

\begin{proof}
Let $\sigma$ be a non-negative measure over $\left[  0,R\right]  .$ Let
$f\in\mathbb{C}\left[  x_{1},..,x_{d}\right]  $ and let $f_{k,l}$ be the
Laplace-Fourier coefficient of $f.$ By Proposition \ref{PropTTT} it is clear
that
\[
T\left(  f\right)  :=\int_{0}^{R}f_{1,1}\left(  r\right)  r^{-1}d\sigma\left(
r\right)
\]
is pseudo-positive definite. We take now for $\sigma$ the Dirac functional at
$r=0$. Suppose that $T$ has a signed representing measure $\mu$ which is
pseudo-positive. Then the measure $\mu_{11}$ is non-negative, and it is
defined by the equation $\int_{0}^{\infty}h\left(  r\right)  d\mu_{11}\left(
r\right)  :=\int_{\mathbb{R}^{n}}h\left(  \left|  x\right|  \right)
Y_{11}\left(  x\right)  d\mu$ for any continuous function $h:\left[
0,\infty\right)  \rightarrow\mathbb{C}$ with compact support. Take now
$h\left(  r\right)  =r^{2}.$ Then by Proposition \ref{pseudopos}
\[
\int_{0}^{\infty}r^{2}d\mu_{11}\left(  r\right)  =\int_{\mathbb{R}^{n}}\left|
x\right|  ^{2}Y_{11}\left(  x\right)  d\mu=T\left(  \left|  x\right|
^{2}Y_{11}\left(  x\right)  \right)  =0.
\]
It follows that $\mu_{11}$ has support $\left\{  0\right\}  .$ On the other
hand, if we take a sequence of functions $h_{m}\in C_{c}\left(  \left[
0,\infty\right)  \right)  $ such that $h_{m}\rightarrow1_{\left\{  0\right\}
},$ then we obtain
\[
\mu_{11}\left(  \left\{  0\right\}  \right)  =\lim_{m\rightarrow\infty}%
\int_{\mathbb{R}^{n}}h_{m}\left(  \left|  x\right|  \right)  Y_{11}\left(
x\right)  d\mu.
\]
But $h_{m}\left(  \left|  x\right|  \right)  Y_{11}\left(  x\right)  $
converges to the zero-function, and Lebesgue's theorem shows that $\mu
_{11}\left(  \left\{  0\right\}  \right)  =0,$ so $\mu_{11}=0.$ This is a
contradiction since
\[
\int_{0}^{\infty}1d\mu_{11}\left(  r\right)  =\int_{\mathbb{R}^{n}}%
Y_{11}\left(  x\right)  d\mu=T\left(  Y_{11}\right)  =\int_{0}^{R}%
1d\sigma\left(  r\right)  =1.
\]
The proof is complete.
\end{proof}

ACKNOWLEDGMENT. Both authors acknowledge the support of the Institutes
Partnership Project with the Alexander von Humboldt Foundation, Bonn. The
second author is supported in part by Grant MTM2006-13000-C03-03 of the D.G.I.
of Spain.

Author's addresses:

1. Ognyan Kounchev, Institute of Mathematics, University of Duisburg-Essen,
Lotharstr. 65, 47057 Duisburg, Germany; Institute of Mathematics and
Informatics, Bulgarian Academy of Sciences, 8 Acad. G. Bonchev Str., 1113
Sofia, Bulgaria;

e--mail: kounchev@math.bas.bg; kounchev@gmx.de

2. Hermann Render, Departamento de Matem\'{a}ticas y Computati\'{o}n,
Universidad de la Rioja, Edificio Vives, Luis de Ulloa, s/n. 26004
Logro\~{n}o, Spain; e-mail: render@gmx.de; herender@unirioja.es
\end{document}